\documentclass[11pt]{amsart}

\usepackage{amsfonts, amsthm}
\usepackage{amssymb, amscd, enumerate}

\usepackage[margin=4cm]{geometry} 
\usepackage{hyperref,amsmath, graphicx}

\def\Z{\mathbb{Z}}
\def\uA{\mathcal{A}}
\def\uB{\mathcal{B}}
\def\uC{\mathcal{C}}
\def\uD{\mathcal{D}}
\def\uL{\mathcal{L}}
\def\uuL{\underline{\mathcal{L}}}

\def\Hom{\mathrm{Hom}}
\def\id{\mathrm{id}}

\def\uCpi{\uC_{\pi}^+}

\def\C{\textbf{k}}
\def\e{\textbf{e}}
\def\svec{\mathrm{sVect}}
\def\usvec{\underline{\mathrm{sVect}}}

\def\euCpi{\underline{\uC_\pi^+}}

\def\vecgw{\mathrm{Vec}_G^\tau}
\def\ev{\mathrm{ev}}
\def\coev{\mathrm{coev}}
\def\unit{\textbf{1}}
\def\ounit{\textbf{1}^{\underline{1}}}
\def\eunit{\textbf{1}^{\underline{0}}}
\def\xa{X^{\underline{a}}}

\def\iso{\xrightarrow{\sim}}
\def\whuL{\widehat{\left(\underline{\uL}\right)}}
\def\uI{\mathcal{I}}
\def\sgr{\mathrm{sGr}(\uC)}
\def\sgrr{\mathrm{sGr}}

\usepackage{tabularx}
\newcolumntype{Y}{>{\centering\arraybackslash}X}

\newtheorem{theorem}{Theorem}[section]
\newtheorem*{theorem*}{Theorem}
\newtheorem*{example*}{Example}
\newtheorem{proposition}[theorem]{Proposition}
\newtheorem{corollary}[theorem]{Corollary}
\newtheorem{lemma}[theorem]{Lemma}

\theoremstyle{definition}
\newtheorem{definition}[theorem]{Definition}
\newtheorem{example}[theorem]{Example}
\newtheorem{remark}[theorem]{Remark}

\usepackage{titletoc}

\begin{document}

\title{Fermionic 6j-symbols in superfusion categories}
\author{Robert Usher}

\begin{abstract}
We describe how the study of \emph{superfusion categories} (roughly speaking, fusion categories enriched over the category of super vector spaces)
reduces to that of \emph{fusion categories over $\usvec$}, in the sense of \cite{DGNO:BFC}.  Following \cite{BE:2016}, we give
the construction of the \emph{underlying fusion category} of a superfusion category, and give an explicit formula for the associator
in this category in terms of 6$j$-symbols.  We give a definition of the $\pi$-Grothendieck ring of a superfusion category,
and prove a version of Ocneanu rigidity for superfusion categories.
\end{abstract}

\maketitle

\section{Introduction}
In condensed matter physics, the use of fusion categories to construct topological quantum field theories
is reasonably well understood.  In \cite{TV:Q6J} and \cite{T:QIK3},  Turaev and Viro constructed
invariants of 3-manifolds from quantum 6$j$-symbols, and showed that these lead to a 3-dimensional
non-oriented topological quantum field theory.  Barrett and Westbury \cite{BW:IPL3M} showed
that these invariants can be constructed from any spherical fusion category.  Following this,
Kirillov and Balsam \cite{KB:TVIETQFT}, and Turaev and Virelizier \cite{TV:two3dTQFT} proved that
the Turaev-Viro-Barrett-Westbury invariants of a spherical fusion category $\uA$ are the same
as the Reshetikhin-Turaev invariants \cite{RT:Invariants} derived from $\mathcal{Z}(\uA)$.

More recently, Douglas, Schommer-Pries and Snyder \cite{DSPS:DTC} showed that fusion categories are fully dualizable objects
in the 3-category of monoidal categories, and so by the cobordism hypothesis \cite{Lurie:TFT} we can
associate a fully local 3-dimensional TQFT to any fusion category.

Gaiotto and Kapustin \cite{GK:2015}, following the work of Gu, Wang and Wen \cite{GWW:2010} described a fermionic
analogue of the Turaev-Viro construction whose initial data is a spherical \emph{superfusion category},
and Bhardwaj, Gaiotto and Kapustin \cite{BGK:sTTF} have further studied spin-TQFTs.  
In comparison to the fusion category case however, not much is known about how to construct TQFTs using superfusion categories.

A \emph{superfusion category} over $\C$ is a semisimple rigid monoidal supercategory (i.e. a category enriched
over $\usvec$) with finitely many simple objects
and finite dimensional superspaces of morphisms, with simple unit object.  In particular, the collection of morphisms
between objects forms a super vector space, and the tensor product of morphisms satisfies the \emph{super interchange law}
\[
(f \otimes g) \circ (h \otimes k) = (-1)^{|g||h|} (f \circ h) \otimes (g \circ k)
\]
Following \cite{GWW:2010}, a simple object $X$ is called \emph{Bosonic} if $\mathrm{End}(X) \simeq \C^{1|0}$,
and \emph{Majorana} if $\mathrm{End}(X) \simeq \C^{1|1}$.   A superfusion category is called
\emph{Bosonic} if all of its simple objects are Bosonic.  Since the unit object in any
superfusion category is necessarily Bosonic, there are no Majorana superfusion categories.

In this paper, we give the construction of the \emph{underlying fusion category} of a superfusion category, using a construction described by Brundan and Ellis \cite{BE:2016}.  The underlying fusion category of a superfusion category is naturally endowed with the structure of a fusion category over $\usvec$ (in the sense of \cite[Definition 7.13.1]{DGNO:BFC}), the category of super vector spaces together with the even linear maps between them.

The associator in a semisimple tensor category (in particular, a fusion category) admits a description in terms of \emph{6$j$-symbols} satisfying a version of the pentagon equation, see i.e. \cite{T:QIK3}, \cite{W:TQC}. In a similar way, the associator in a superfusion category can be described
in terms of \emph{fermionic 6$j$-symbols} satisfying the \emph{super pentagon equation} \cite{GWW:2010}.

The main goal of this paper is to describe the relation between a superfusion category and its underlying fusion category.
More precisely, we give an explicit formula for the 6$j$-symbols of the underlying fusion category in terms
of the fermionic 6$j$-symbols of the superfusion category, and show that these 6$j$-symbols satisfy the pentagon equation.

If $\uC$ is a Bosonic pointed superfusion category, i.e. a Bosonic superfusion category such that the isomorphism classes
of simple objects form a group $G$, then the fermionic 6$j$-symbols in $\uC$ are described by a 3-supercocycle \cite{GWW:2010} $\widetilde{F} : G^3 \to \C^\times$ satisfying
\begin{equation*}
\widetilde{F}(g,h,k) \widetilde{F}(g,hk,l) \widetilde{F}(h,k,l) = (-1)^{\omega(g,h) \omega(k,l)} \widetilde{F}(gh,k,l) \widetilde{F}(g,h,kl)
\end{equation*}
where $\omega \in H^2(G,\Z/2\Z)$ is a 2-cocycle on $G$.  In this situation, our formula for the 6$j$-symbols on the underlying
fusion category gives a 3-cocycle on the $\Z/2\Z$-central extension of $G$ determined by $\omega$, whose restriction to $G$
is $\widetilde{F}$.  In particular, this implies that every 3-supercocycle on $G$ arises as the restriction of a (genuine)
3-cocycle on a central extension of $G$ by $\Z/2\Z$.

We also define the \emph{$\pi$-Grothendieck ring} $\sgr$ of a superfusion category $\uC$,
which is an algebra over $\Z^\pi := \Z[\pi]/(\pi^2-1)$, and describe the relation between
the $\pi$-Grothendieck ring of $\uC$ and the Grothendieck ring of the underlying fusion category $\euCpi$.
As a corollary of this, we deduce a version of Ocneanu rigidity for superfusion categories.

\section{Fusion categories}
Let $\C$ denote an algebraically closed field of characteristic 0.

\begin{definition}[\cite{ENO:OFC}]
A \emph{fusion category} over $\C$ is a semisimple rigid $\C$-linear monoidal category $\uA$ with finitely many
isomorphism classes of simple objects and finite-dimensional spaces of morphisms such that the unit object is simple.
\end{definition}

In this section, we recall how the associator
\[
a : (- \otimes -) \otimes - \iso - \otimes (- \otimes -).
\]
in a fusion category can be described in terms of \emph{6$j$-symbols}, closely following the discussion in \cite[Chapter 4]{W:TQC}, see also \cite[Chapter VI]{T:QIK3}.

\begin{example}
The category $\mathrm{Vec}$ of finite-dimensional $\C$-vector spaces is a fusion category.  
More generally, if $G$ is a finite group and $\tau \in H^3(G,\C^\times)$ is a 3-cocycle,
then the category $\mathrm{Vec}_G^\tau$ of $G$-graded finite dimensional $\C$-vector spaces
with associativity defined by $\tau$ is a fusion category.  
\label{ex:fusioncat}
\end{example}

\subsection{6$j$-symbols}
Let $\uA$ be a fusion category, and $X_i$, $i \in I$ representatives of the isomorphism classes
of simple objects in $\uA$.  The monoidal structure on $\uA$ determines the \emph{fusion rules} 
\[
X_i \otimes X_j \simeq \bigoplus_{m \in I} N^{ij}_m X_m
\]
where 
\[
N^{ij}_m = \dim \Hom_{\uA}(X_m, X_i \otimes X_j) = \dim \Hom_{\uA}(X_i \otimes X_j, X_m) \in \Z_{\geq 0}.
\]
is the \emph{multiplicity} of $X_m$ in $X_i \otimes X_j$.  The notion of admissability will be useful.

\begin{definition}[{see \cite[Definition 4.1]{W:TQC}}] \label{def:fusionadmissable}
Let $\uA$ be a fusion category with simple objects indexed by a set $I$.  We say
a triple $(i,j,m) \in I^3$ is \emph{admissable} if $N^{ij}_m > 0$.  A quadruple $(i,j,m,\alpha) \in I^3 \times \Z_{\geq 0}$
is \emph{admissable} if $(i,j,m)$ is admissable, and $1 \leq \alpha \leq N^{ij}_m$.  A decuple $(i,j,m,k,n,t,\alpha,\beta,\eta,\varphi) \in I^6 \times \Z_{\geq 0}^4$ is \emph{admissable} if each fo the quadruples
$(i,j,m,\alpha)$, $(m,k,n,\beta)$, $(j,k,t,\eta)$ and $(i,t,n, \varphi)$ are admissable.
\end{definition}

\begin{remark}
A fusion category is called \emph{multiplicity-free} if $N^{ij}_m \in \{0,1\}$ for all $i, j, m \in I$ \cite[Definition 4.5]{W:TQC}.
In the multiplicity-free case, an admissable decuple is completely described by the sextuple $(i,j,m,k,n,t)$,
in which case this definition recovers \cite[Definition 4.7]{W:TQC}.
\end{remark}

That the triple $(i,j,m)$ is admissable is equivalent to saying that $X_m$ is a direct
summand of $X_i \otimes X_j$.  For each admissable triple $(i,j,m)$, choose a basis for the space
\[
\Hom_{\uA}(X_i \otimes X_j,X_m).
\]
Admissable quadruples of the form $(i,j,m,\alpha)$ then label the basis vectors
of $\Hom_{\uA}(X_i \otimes X_j, X_m)$.  We denote these basis vectors by $\e^{ij}_m(\alpha)$, where $1 \leq \alpha \leq N^{ij}_m$.  We wish to describe
the associator $a(X_i,X_j,X_k) : (X_i \otimes X_j) \otimes X_k \to X_i \otimes (X_j \otimes X_k)$ in terms of 
our chosen basis.  Indeed, fixing admissable quadruples $(i,j,m,\alpha)$ and $(m,k,n,\beta)$, we have the composition
\begin{equation}
(X_i \otimes X_j) \otimes X_k \xrightarrow{\e^{ij}_m(\alpha) \otimes \id_{X_k}} X_m \otimes X_k \xrightarrow{\e^{mk}_n(\beta)} X_n
\label{eq:6j_fusioncat}
\end{equation}
which we may represent graphically as
\begin{center}
\includegraphics{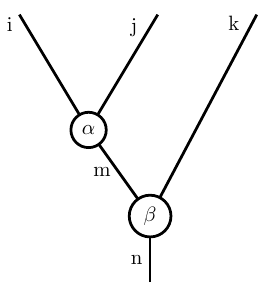}
\end{center}
Let $t \in I$.  If $(j,k,t,\eta)$ and $(i,t,n,\varphi)$ are admissable, then we have the composition
\begin{equation}
(X_i \otimes X_j) \otimes X_k \xrightarrow{a(X_i,X_j,X_k)} X_i \otimes (X_j \otimes X_k) \xrightarrow{\id_{X_i} \otimes \e^{jk}_t(\eta)} X_i \otimes X_t \xrightarrow{\e^{it}_n(\varphi)} X_n
\label{eq:6j_fusioncatrhs}
\end{equation}
which we may represent graphically as
\begin{center}
\includegraphics{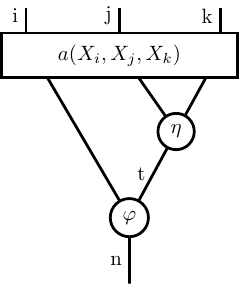}
\end{center}
Fix $i, j, k, n \in I$.  Taking the direct sum of the above compositions over
all $t \in I$ such that $(j,k,t,\eta)$ and $(i,t,n,\varphi)$ are admissable gives an isomorphism \cite[Lemma 1.1.1, Lemma 1.1.2]{T:QIK3}
\begin{align*}
\bigoplus_{t \in I} \Hom_{\uA}(X_j \otimes X_k, X_t) \otimes \Hom_{\uA}(X_i \otimes X_t, X_n) &\iso \Hom_{\uA}((X_i \otimes X_j) \otimes X_k, X_n)\\
\e^{jk}_t(\eta) \otimes \e^{it}_n(\varphi) &\mapsto \e^{it}_n(\varphi) \circ (\id_{X_i} \otimes \e^{jk}_t(\eta)) \circ a(X_i,X_j,X_k)
\end{align*}
Expressing \eqref{eq:6j_fusioncat} in terms of this basis determines a constant $F^{ijm,\alpha\beta}_{knt,\eta\varphi} \in \C$
for each admissable decuple $(i,j,m,k,n,t,\alpha,\beta,\eta,\varphi)$ in $\uA$, defined by the graphical equation:
\begin{center} \label{diag:6j_defpic}
\includegraphics{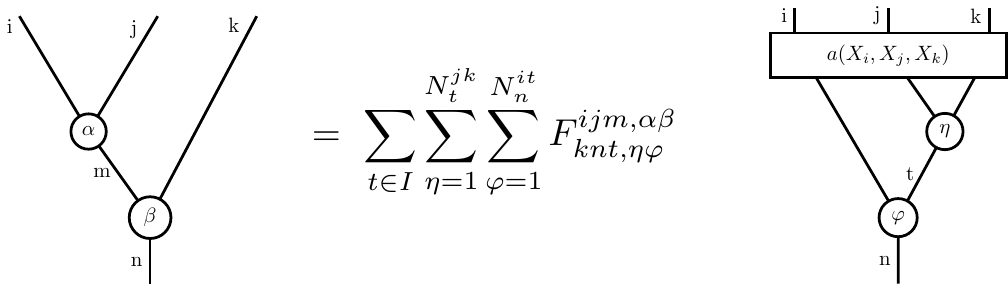}
\end{center}
This describes the associator in $\uA$ as a collection of matrices
\begin{equation*}
F^{ijm}_{knt} : \Hom_{\uA}(X_i \otimes X_j, X_m) \otimes \Hom_{\uA}(X_m \otimes X_k, X_n) \to \Hom_{\uA}(X_j \otimes X_k, X_t) \otimes \Hom_{\uA}(X_i \otimes X_t,X_n)
\end{equation*}
whose entries are the constants defined above.  The matrices $F^{ijm}_{knt}$ are called \emph{6$j$-symbols}, as they depend on six indices.  
If $(i,j,m,k,n,t,\alpha,\beta,\eta,\varphi)$ is not admissable, then by convention we set $F^{ijm,\alpha\beta}_{knt,\eta\varphi} = 0$. 
The pentagon axiom in $\uA$ is then equivalent to the following equation in terms of 6$j$-symbols.

\begin{lemma}[Pentagon equation]
Let $\uA$ be a fusion category with simple objects indexed by a set $I$.  If $i,j, k,l,m,n,t,p,q,s \in I$ 
and $\alpha,\beta,\eta,\chi,\gamma,\delta,\phi \in \Z_{\geq 0}$.  Then
\begin{equation}
\sum_{t \in I} \sum_{\eta=1}^{N^{jk}_t} \sum_{\varphi=1}^{N^{it}_n} \sum_{\kappa=1}^{N^{tl}_s} F^{ijm,\alpha\beta}_{knt,\eta\varphi} F^{itn,\varphi\chi}_{lps,\kappa\gamma} F^{jkt,\eta\kappa}_{lsq,\delta\phi} = \sum_{\epsilon=1}^{N_p^{mq}} F^{mkn,\beta\chi}_{lpq,\delta\epsilon} F^{ijm,\alpha\epsilon}_{qps,\delta\gamma}
\label{def:6j_pentagon}
\end{equation}
\end{lemma}

\begin{example}[see Example 2.3.8 \cite{ENO:2015}] \label{ex:ptedfusioncase}
Continuing Example \ref{ex:fusioncat}, the fusion category $\vecgw$ has pairwise non-isomorphic
simple objects $\{\delta_g\}_{g \in G}$ satisfying $\delta_g \otimes \delta_h \simeq \delta_{gh}$,
so admissable quadruples are of the form $(g,h,gh,1)$ for all $g,h \in G$.  Thus given $g,h,k \in G$
we can write $F(g,h,k) := F^{gh}_k \in \C^\times$ for the corresponding 6$j$-symbol unambiguously.  The pentagon equation \eqref{def:6j_pentagon} then reduces to
\begin{equation*}
F(g,h,k) F(g,hk,l) F(h,k,l) = F(gh,k,l) F(g,h,kl)\quad g, h, k, l \in G
\end{equation*}
so $F$ is a 3-cocycle on $G$ with values in $\C^\times$.
\end{example}

\section{Superfusion categories}
In this section we recall the definition of a superfusion category using the language of \cite{BE:2016}, and
describe the associator in a superfusion category in terms of \emph{fermionic 6$j$-symbols}, following \cite{GK:2015}.  
By a superspace we always mean a $\Z/2\Z$-graded $\C$-vector space $V$.  The parity of a homogeneous element $v \in V$
will be denoted by $|v|$.

\begin{definition}
Let $\usvec$ be the category whose objects are superspaces, and whose morphisms are even linear maps,
i.e. linear maps preserving the grading.  
\end{definition}

We can make $\usvec$ into a monoidal category by defining
the tensor product of superspaces $V$ and $W$ to be the superspace $V \otimes W$ with
$(V \otimes W)_0 := (V_0 \otimes W_0) \oplus (V_1 \otimes W_1)$ and $(V \otimes W)_1 := (V_1 \otimes W_0) \oplus (V_0 \otimes W_1)$,
with the tensor product of morphisms defined in the obvious way.  The braiding
\begin{equation*}
c_{V,W}(v \otimes w) = (-1)^{|v||w|} v \otimes w
\label{eq:usvecbraiding}
\end{equation*}
defined on homogeneous $v \in V$ and $w \in W$ makes $\usvec$ into a symmetric monoidal category.

\begin{definition}[{see \cite[Definition 1.1]{BE:2016} and \cite[Section 1.2]{Kelly:ECT} for details}]
A \emph{supercategory} is a $\usvec$-enriched category.  A \emph{superfunctor} between supercategories
is a $\usvec$-enriched functor, and a \emph{supernatural transformation} between superfunctors is a $\usvec$-enriched supernatural transformation. 
We say a supernatural transformation is \emph{even} if all its component maps are even.
\end{definition}

In particular, if $\mathcal{A}$ is a supercategory, then $\Hom_{\mathcal{A}}(X,Y)$ is a superspace for all $X$, $Y \in \uA$,
and composition
\[
\Hom_{\mathcal{A}}(Z,Y) \otimes \Hom_{\mathcal{A}}(X,Y) \to \Hom_{\mathcal{A}}(X,Z)
\]
is an \emph{even} linear map for all $X, Y, Z \in \uA$.

\begin{remark}
Given supercategories $\mathcal{A}$ and $\mathcal{B}$, we can form their tensor product $\mathcal{A} \boxtimes \mathcal{B}$.  Objects
of $\uA \boxtimes \uB$ are pairs $(X,Y)$ with $X \in \uA$ and $Y \in \mathcal{B}$.  Morphisms in $\uA \boxtimes \uB$
are given by $\Hom_{\mathcal{A} \boxtimes \mathcal{B}}((X,Y),(W,Z)) := \Hom_{\mathcal{A}}(X,W) \otimes \Hom_{\mathcal{B}}(Y,Z)$, with composition in $\mathcal{A} \boxtimes \mathcal{B}$ defined using the braiding in $\usvec$, see \cite{BE:2016} for details.
\end{remark}

\begin{definition}[{\cite[Definition 1.4]{BE:2016}}]
A \emph{monoidal supercategory} is a supercategory $\mathcal{D}$, together with a tensor product superfunctor $- \otimes - : \mathcal{D} \boxtimes \mathcal{D} \to \mathcal{D}$, a unit object $\unit$, and even supernatural isomorphisms $a : (- \otimes -) \otimes - \iso - \otimes (- \otimes -)$, $l : \unit \otimes - \iso -$ and $r : - \otimes \unit \iso -$ satisfying axioms analogous to the ones of a monoidal category.
A \emph{monoidal superfunctor} between monoidal supercategories $\mathcal{D}$ and $\mathcal{E}$ is a superfunctor
$F : \mathcal{D} \to \mathcal{E}$ such that $F(\unit_{\mathcal{D}})$ is evenly isomorphic
to $\unit_{\mathcal{E}}$, together with even coherence maps $J : F(-) \otimes F(-) \to F(- \otimes -)$
satisfying the usual axioms.
\end{definition}

An important feature of monoidal supercategories is the \emph{super interchange law}
\begin{equation*}
(f \otimes g) \circ (h \otimes k) = (-1)^{|g||h|} (f \circ h) \otimes (g \circ k)
\label{eq:SIL}
\end{equation*}
describing the composition of tensor products of morphisms.  We recall the following definitions from \cite[Appendix C]{GWW:2010}.

\begin{definition}
A \emph{superfusion category} over $\C$ is a semisimple rigid monoidal supercategory $\uC$
with finitely many simple objects and finite dimensional superspaces of morphisms such
that the unit object $\unit$ is simple.  A simple object $X \in \uC$ is \emph{Bosonic} if $\mathrm{Hom}_{\uC}(X,X) \simeq \C^{1|0}$,
and \emph{Majorana} if $\mathrm{End}_{\uC}(X) \simeq \C^{1|1}$.  A superfusion category is called \emph{Bosonic} if all its simple objects are Bosonic.
\end{definition}

The unit object $\unit$ in a superfusion category $\uC$ is always Bosonic. Indeed, since $\unit \otimes \unit \simeq \unit$,
the tensor product functor induces an embedding
\[
\Hom_{\uC}(\unit,\unit) \otimes \Hom_{\uC}(\unit,\unit) \to \Hom_{\uC}(\unit \otimes \unit, \unit \otimes \unit) \simeq \Hom_{\uC}(\unit,\unit)
\]
which implies $\Hom_{\uC}(\unit,\unit) \simeq \C^{1|0}$.

\begin{remark}
That $\uC$ is rigid means that for each $X \in \uC$ we have a left dual $X^{*} \in \uC$ and a right dual $^{*}X \in \uC$,
together with even morphisms $\ev_X : X^* \otimes X \to \textbf{1}$, $\coev_X : \textbf{1} \to X \otimes X^*$, $\ev'_X : X \otimes ^{*}X \to \textbf{1}$,
and $\coev'_X : \textbf{1} \to ^{*}X \otimes X$ satisfying the usual equations, see \cite[Section 2.10]{ENO:2015} for details.
\end{remark}

\subsection{Fermionic 6$j$-symbols}
Let $\uC$ be a superfusion category, and $X_i$, $i \in I$ representatives of the isomorphism
classes of simple objects in $\uC$.  The monoidal structure on $\uC$
determines the \emph{superfusion rules}
\[
X_i \otimes X_j \simeq \bigoplus_{m \in I} N^{ij}_m X_m
\]
where 
\[
N^{ij}_m = \dim \Hom_{\uC}(X_i \otimes X_j, X_m) = \dim \Hom_{\uC}(X_m, X_i \otimes X_j) \in \Z_{\geq 0}
\]
i.e. $N^{ij}_m$ is the dimension of the superspace $\Hom_{\uC}(X_i \otimes X_j, X_m)$.  With this notation, our notion of admissable triple, quadruple, and decuple remain the same as in Definition \ref{def:fusionadmissable}.  As in the fusion category case, for each admissable triple $(i,j,m)$ we choose a homogeneous basis for the superspace $\Hom_{\uC}(X_i \otimes X_j, X_m)$ denoted by $\e^{ij}_m(\alpha)$, where $1 \leq \alpha \leq N^{ij}_m$.  Let $s^{ij}_m(\alpha) = |\e^{ij}_m(\alpha)|$ denote the parity of the corresponding basis vector.  

\begin{definition}\label{def:parityadmissable}
We say that an admissable decuple $(i,j,m,k,n,t,\alpha,\beta,\eta,\varphi)$
is \emph{parity admissable} if
\begin{equation}
s^{ij}_m(\alpha) + s^{mk}_n(\beta) = s^{jk}_t(\eta) + s^{it}_n(\varphi).
\label{eq:defparityadmissable}
\end{equation}
\end{definition}
In exactly the same way as in the fusion category case, we define constants $\widetilde{F}^{ijm,\alpha\beta}_{knt,\eta\varphi} \in \C$
for each admissable decuple $(i,j,m,k,n,t,\alpha,\beta,\eta,\varphi)$ in $\uC$, defined by the graphical equation

\begin{center} \label{diag:f6j_defpic}
\includegraphics{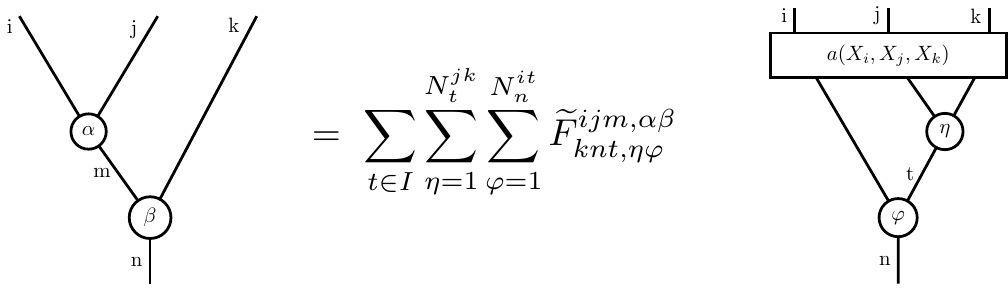}
\end{center}
\begin{remark}
We recover the parity admissability condition \eqref{eq:defparityadmissable} by comparing the parity
of both sides of the above equation.  In particular, the constant $\widetilde{F}^{ijm,\alpha\beta}_{knt,\eta\varphi}$
is non-zero only for parity admissable decuples $(i,j,m,k,n,t,\alpha,\beta,\eta,\varphi)$.
\end{remark}
This describes the associativity constraint in $\uC$ as a collection of invertible matrices
\[
\widetilde{F}^{ijm}_{knt} : \Hom_{\uC}(X_i \otimes X_j, X_m) \otimes \Hom_{\uC}(X_m \otimes X_k ,X_n) \to \Hom_{\uC}(X_j \otimes X_k, X_t) \otimes \Hom_{\uC}(X_i \otimes X_t,X_n)
\]
whose entries are the constants defined above.  The matrices $\widetilde{F}^{ijm}_{knt}$ are called \emph{fermionic 6$j$-symbols}.  If $(i,j,m,k,n,t,\alpha,\beta,\eta,\varphi)$ is not (parity) admissable, then by convention we set $\widetilde{F}^{ijm,\alpha\beta}_{knt,\eta\varphi} = 0$. The super pentagon axiom in $\uC$ is equivalent to the following
equation in terms of fermionic 6$j$-symbols, called the \emph{fermionic pentagon identity} in \cite{GWW:2010}.

\begin{lemma}[Super pentagon equation]
Let $\uC$ be a superfusion category with simple objects indexed by a set $I$.  If $i,j, k,l,m,n,t,p,q,s \in I$ 
and $\alpha,\beta,\eta,\chi,\gamma,\delta,\phi \in \Z_{\geq 0}$, then
\begin{equation}
\sum_{t \in I} \sum_{\eta=1}^{N^{jk}_t} \sum_{\varphi=1}^{N^{it}_n} \sum_{\kappa=1}^{N^{tl}_s} \widetilde{F}^{ijm,\alpha\beta}_{knt,\eta\varphi} \widetilde{F}^{itn,\varphi\chi}_{lps,\kappa\gamma} \widetilde{F}^{jkt,\eta\kappa}_{lsq,\delta\phi} = (-1)^{s^{ij}_m(\alpha)s^{kl}_q(\delta)} \sum_{\epsilon=1}^{N_p^{mq}} \widetilde{F}^{mkn,\beta\chi}_{lpq,\delta\epsilon} \widetilde{F}^{ijm,\alpha\epsilon}_{qps,\delta\gamma}
\label{def:6j_spentagon}
\end{equation}
\end{lemma}

\begin{example} \label{ex:pted_1}
We say a superfusion category $\uC$ is \emph{pointed} if any simple object $X \in \uC$ is invertible, that is, there
exists $Y \in \uC$ such that $X \otimes Y \simeq Y \otimes X \simeq \textbf{1}$. Let $\uC$ be a Bosonic superfusion
category, and let $G$ be the (finite) group of isomorphism
classes of simple objects in $\uC$, and choose $X_g$, $g \in G$ a set of representatives of simple objects in $\uC$.  
Then $X_{g} \otimes X_{h} \simeq X_{gh}$ for all $g, h \in G$, so admissable
quadruples in $\uC$ are of the form $(g,h,gh,1)$ for all $g, h \in G$.  Let $\omega(g,h)$ denote the parity of the one-dimensional superspace
$\Hom_{\uC}(X_g \otimes X_h, X_{gh})$, then the parity admissability condition \eqref{eq:defparityadmissable} implies
\begin{equation*}
\omega(g,h) + \omega(gh,k) = \omega(h,k) + \omega(g,hk).
\label{eq:parity2cocyclecondition}
\end{equation*}
for all $g, h, k \in G$, so $\omega$ is a 2-cocycle on $G$ with values in $\Z/2\Z$. The super pentagon equation \eqref{def:6j_spentagon} implies
\begin{equation*}
\widetilde{F}(g,h,k) \widetilde{F}(g, hk, l) \widetilde{F}(h,k,l) = (-1)^{\omega(g,h) \omega(k,l)} \widetilde{F}(gh, k,l) \widetilde{F}(g,h,kl)
\label{eq:pted_spe}
\end{equation*}
for all $g, h, k, l \in G$, so following \cite{GWW:2010} we say $\widetilde{F}$ is a \emph{3-supercocycle} on $G$.
\end{example}

\section[Fusion categories over sVect]{Fusion categories over sVect}
In this section, we show that every superfusion category is equivalent to a \emph{$\Pi$-complete superfusion category}
(i.e. a superfusion category equipped with an odd isomorphism $\zeta : \pi \iso \unit$), and give the construction
of the \emph{underlying fusion category} of a $\Pi$-complete superfusion category, following \cite{BE:2016}.

Recall that a fusion category is \emph{braided} if it is equipped with a natural isomorphism $c_{X,Y} : X \otimes Y \iso Y \otimes X$ satisfying well-known axioms, see \cite[Definition 2.1]{JS:BTC}, \cite[Definition 8.1.1]{ENO:2015}.  A monoidal
functor between braided fusion categories is \emph{braided} if it respects the braiding, see \cite[Definition 2.3]{JS:BTC}, \cite[Definition 8.1.7]{ENO:2015}.  

\begin{definition}[{\cite[Definition 7.13.1]{ENO:2015}}] \label{def:drinfeldcentre}
The centre of a fusion category $\uA$ is the category $\mathcal{Z}(\uA)$ whose
objects are pairs $(Z,\beta)$ where $Z \in \uA$ and
\begin{equation*}
\beta_X : X \otimes Z \iso Z \otimes X,\ X \in \uA
\end{equation*}
is a natural isomorphism such that the following diagram
\begin{center}
\includegraphics{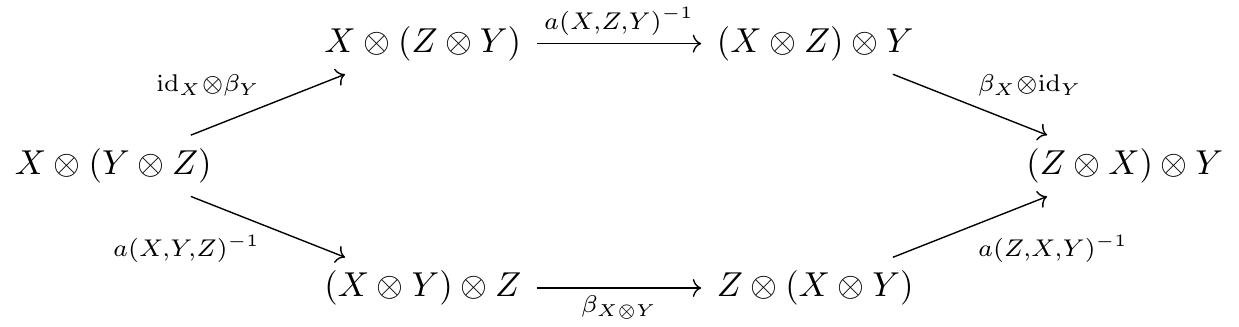}
\end{center}
is commutative for all $X, Y \in \uA$.

A morphism from $(Z,\beta)$ to $(Z',\beta')$ is a morphism $f \in \Hom_{\uA}(Z,Z')$ such that
\begin{equation*}
(f \otimes \id_X) \circ \beta_X = \beta_X' \circ (\id_X \otimes f)
\end{equation*}
for all $X \in \uA$.
\end{definition}
Equipping $\mathcal{Z}(\uA)$ with the usual tensor product (see \cite[Definition 7.13.1]{ENO:2015}) and braiding
$c_{(Z,\beta),(Z',\beta')} := \beta'_Z$ makes $\mathcal{Z}(\uA)$ into a braided fusion category, see \cite[Proposition 8.5.1 and Theorem 9.3.2]{ENO:2015}.

\begin{definition}[{\cite[Definition 4.16]{DGNO:BFC}}]\label{def:fcoverusvec}
A fusion category over $\usvec$ is a fusion category $\uA$ equipped with a braided
functor $\usvec \to \mathcal{Z}(\uA)$.  Equivalently, this is an object $(\pi,\beta)$ in the centre
$\mathcal{Z}(\uA)$ together with an even isomorphism $\xi : \pi \otimes \pi \iso \unit$ such that
\begin{equation}
(\xi^{-1} \otimes \id_X) \circ l^{-1}_X \circ r_X \circ (\id_X \otimes \xi) = a(\pi,\pi,X)^{-1} \circ (\id_\pi \otimes \beta_X) \circ a(\pi,X,\pi) \circ (\beta_X \otimes \id_\pi) \circ a(X,\pi,\pi)^{-1}
\label{eq:fcoverxi}
\end{equation}
for all $X \in \uA$, and
\begin{equation}
\beta_\pi = -\id_{\pi \otimes \pi} \in \Hom_{\uA}(\pi \otimes \pi, \pi \otimes \pi).
\label{eq:fcovergamma}
\end{equation}
In this situation we say $(\uA,\pi,\beta,\xi)$ is a \emph{fusion category over $\usvec$}.
\end{definition}
In the language of \cite{BE:2016},  the quadruple $(\uA,\pi,\beta,\xi)$ is an example of a \emph{monoidal $\Pi$-category}.

\begin{remark}
In this section, we will often draw commutative diagrams with associativity and unit isomorphisms omitted, unless confusion is possible.
For example, we represent Equation \eqref{eq:fcoverxi} as the diagram
\begin{center}
\includegraphics{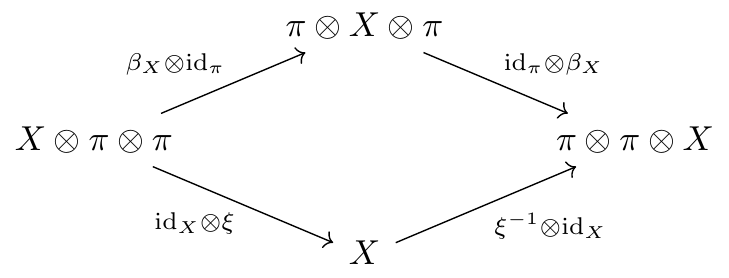}
\end{center}
In addition we say that the diagram
\begin{center}
\includegraphics{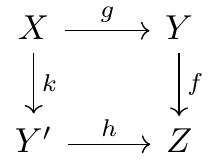}
\end{center}
is \emph{supercommutative} if $h \circ k = - f \circ g$.
\end{remark}

\subsection{The $\Pi$-envelope of a superfusion category}
\begin{definition}
Let $\uC$ be a superfusion category, together with an object $\pi$
and an odd isomorphism $\zeta : \pi \iso \unit$.  In this situation, we say that $(\uC,\pi,\zeta)$ is a \emph{$\Pi$-complete superfusion category}.
\end{definition}

In particular, this implies that every object in $\uC$ is the target of an odd isomorphism.  It turns out that every superfusion category is equivalent to a $\Pi$-complete superfusion category, by the following construction described in \cite{BE:2016}.

\begin{definition}[{see \cite[Definition 1.16]{BE:2016}}] \label{def:pienvelope}
Let $\uC$ be a superfusion category. The \emph{$\Pi$-envelope} of $\uC$ is the rigid monoidal supercategory $\uC_\pi$
with objects of the form $X^{\underline{a}}$, where $X \in \uC$ and $\underline{a} \in \Z/2\Z$, and morphisms defined by
\[
\Hom_{\uC_\pi}(X^{\underline{a}},Y^{\underline{b}})_{\underline{c}} := \Hom_{\uC}(X,Y)_{\underline{a}+\underline{b}+\underline{c}}
\]
If $f : X \to Y$ is a homogeneous morphism in $\uC$ with parity $|f|$, then let $f_{\underline{a}}^{\underline{b}}$ denote the corresponding
morphism $X^{\underline{a}} \to Y^{\underline{b}}$ which has parity $\underline{a} + \underline{b} + |f|$ in $\uC_\pi$.  The composition
in $\uC_\pi$ is induced by the composition in $\uC$, and the tensor proper of objects and morphisms is defined by
\begin{align*}
X^{\underline{a}} \otimes Y^{\underline{b}} &:= (X \otimes Y)^{\underline{a}+\underline{b}}\\
f_{\underline{a}}^{\underline{b}} \otimes g_{\underline{c}}^{\underline{d}} &:= (-1)^{(\underline{c}+\underline{d}+|g|)+d|f|} (f \otimes g)_{\underline{a}+\underline{c}}^{\underline{b}+\underline{d}}
\end{align*}
The unit object of $\uC_\pi$ is $\eunit$, and the maps $a, l,$ and $r$ in $\uC$ extend to $\uC_\pi$ in the obvious way.  
The left dual of an object $\xa \in \uC_\pi$ is given by $(X^*)^{\underline{a}}$, where evaluation
and coevaluation morphisms are given by
\[
\ev_{X^{\underline{a}}} := (\ev_X)_{\underline{0}}^{\underline{0}} : (X^*)^{\underline{a}} \otimes X^{\underline{a}} \to \unit^{\underline{0}}
\]
and
\[
\coev_{X^{\underline{a}}} := (\coev_X)_{\underline{0}}^{\underline{0}} : \unit^{\underline{0}} \to X^{\underline{a}} \otimes (X^*)^{\underline{a}}
\]
Similarly, the right dual of $X^{\underline{a}} \in \uC_\pi$ is $(^{*}X)^{\underline{a}} \in \uC_\pi$,
where $\ev'_{X^{\underline{a}}} := (\ev'_X)^{\underline{0}}_{\underline{0}}$ and $\coev'_{X^{\underline{a}}} := (\coev'_X)^{\underline{0}}_{\underline{0}}$.
\end{definition}
The functor $J : \uC \to \uC_\pi$
sending $X \mapsto X^{\underline{0}}$ and $f \mapsto (f)_{\underline{0}}^{\underline{0}}$ is full, faithful,
and essentially surjective, so $\uC$ and $\uC_\pi$ are equivalent as superfusion categories.  However $J$ need not be a \emph{superequivalence} in general, indeed, in \cite[Lemma 4.1]{BE:2016} it is shown that $J$ is a superequivalence if and only if $\uC$ is $\Pi$-complete.

\begin{definition} \label{def:saenvelope}
The \emph{superadditive envelope} $\uCpi$ of a superfusion category $\uC$ is the superfusion category obtained by taking the additive envelope of the $\Pi$-envelope of $\uC$.
\end{definition}
In $\uCpi$ we have the odd isomorphism $\zeta := (\id_\unit)_{\underline{1}}^{\underline{0}} : \unit^{\underline{1}} \to \unit^{\underline{0}}$,
so $(\uCpi, \unit^{\underline{1}},\zeta)$ is a $\Pi$-complete superfusion category.

\subsection{The underlying fusion category of a \texorpdfstring{$\Pi$}{Pi}-complete superfusion category}
\begin{definition}
Let $(\uL,\pi,\zeta)$ be a $\Pi$-complete superfusion category.  The \emph{underlying fusion category} $\uuL$ 
of $\uL$ is the fusion category with the same objects as $\uL$, but only the even morphisms.
\end{definition}

Since $(\uL,\pi,\zeta)$ is $\Pi$-complete, we can endow $\uuL$ with the structure of a fusion category over $\usvec$.
Indeed, define the even supernatural transformation ${\beta : - \otimes \pi \iso \pi \otimes -}$
by letting $\beta_X$ be the composition

\begin{equation*}
X \otimes \pi \xrightarrow{\id_X \otimes \zeta} X \otimes \unit \xrightarrow{r_X} X \xrightarrow{l_X^{-1}} \unit \otimes X \xrightarrow{\zeta^{-1} \otimes \id_X} \pi \otimes X
\end{equation*}
for $X \in \uuL$.  It is straightforward to check that $\beta$ is an even supernatural transformation,
and that $(\pi,\beta)$ is an object of the centre $\mathcal{Z}(\uuL)$ of $\uuL$.  Let $\xi = l_1 \circ (\zeta \otimes \zeta) : \pi \otimes \pi \iso \unit$, then $\xi$ is even and thus may be viewed as an isomorphism $\pi \otimes \pi \iso \unit$ in $\uuL$.
The following lemma is a special case of \cite[Lemma 3.2]{BE:2016}.

\begin{lemma}
$(\uuL,\pi,\beta,\xi)$ is a fusion category over $\usvec$.
\label{lem:fcatusvec}
\end{lemma}

\begin{proof}
We must show that equations \eqref{eq:fcoverxi} and \eqref{eq:fcovergamma} hold.  For the former,
observe that $\xi^{-1} = -(\zeta^{-1} \otimes \zeta^{-1}) \circ l_1^{-1}$ by the super interchange law, so it is enough to show that the following diagram

\begin{center}
\includegraphics{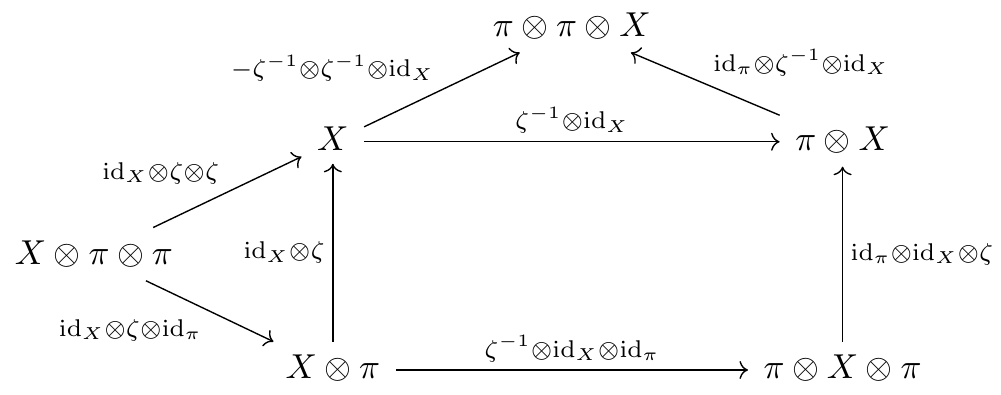}
\end{center}
is commutative.  The super interchange law implies that the top (respectively left) triangle commutes (respectively supercommutes), while supernaturality of $\zeta^{-1}$ implies the rectangle supercommutes, and so the diagram is commutative.  For \eqref{eq:fcovergamma}, consider the following diagram.
\begin{center}
\includegraphics{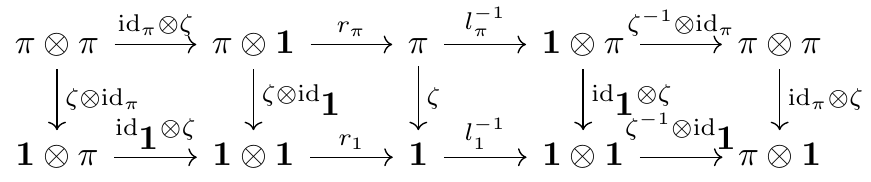}
\end{center}
The first and last cell supercommute by the super interchange law, while the middle cells commute by naturality of $l$ and $r$.  Thus
\begin{equation*}
(\id_\pi \otimes \zeta) \circ \beta_\pi = \left[-(\zeta^{-1} \otimes \id_\unit)\right] \circ \left[-(\id_\unit \otimes \zeta) \circ (\zeta \otimes \id_\pi)\right]
\end{equation*}
and so $\beta_\pi = (\id_\pi \otimes \zeta^{-1}) \circ (\zeta^{-1} \otimes \id_\unit) \circ (\id_\unit \otimes \zeta) \circ (\zeta \otimes \id_\pi) = -\id_{\pi \otimes \pi}$ by the super interchange law.
\end{proof}
Thus to every $\Pi$-complete superfusion category there is a corresponding fusion category over $\usvec$.
We now present the inverse construction, as given in \cite[\S 5]{BE:2016}.

\begin{definition}
Let $(\uA,\pi,\beta,\xi)$ be a fusion category over $\usvec$.  The \emph{associated superfusion category} $\widehat{\uA}$ is the $\Pi$-complete superfusion category with the same objects as $\uA$, but with morphisms defined by
\begin{equation*}
\Hom_{\widehat{\uA}}(X,Y)_0 := \Hom_{\uA}(X,Y) \text{ and } \Hom_{\widehat{\uA}}(X,Y)_1 := \Hom_{\uA}(X, \pi \otimes Y)
\end{equation*}
Let $f \in \Hom_{\widehat{\uA}}(X,Y)$ and $g \in \Hom_{\widehat{\uA}}(Y,Z)$ be homogeneous morphisms in $\uA$, then
their composition $g \widehat{\circ} f$ in $\widehat{\uA}$ is defined in the obvious way, except when
$f$ and $g$ are both odd, in which case $g \widehat{\circ} f$ is induced by the composition

\begin{equation*}
X \xrightarrow{f} \pi \otimes Y \xrightarrow{\id_\pi \otimes g} \pi \otimes (\pi \otimes Z) \xrightarrow{a(\pi,\pi,Z)^{-1}} (\pi \otimes \pi) \otimes Z \xrightarrow{\xi \otimes \id_Z} \unit \otimes Z \xrightarrow{l_Z} Z.
\end{equation*}
The tensor product of objects in $\widehat{\uA}$ is identical to that in $\uA$.  If $f \in \Hom_{\widehat{\uA}}(W,Y)$, $g \in \Hom_{\widehat{\uA}}(X,Z)$ are homogeneous morphisms, then their tensor product $f \widehat{\otimes} g : W \otimes X \to Y \otimes Z$ is defined as follows:

\begin{itemize}
\item If $f$ and $g$ are both even, let $f \widehat{\otimes} g := f \otimes g$.
\item If $f$ is even and $g$ is odd, let $f \widehat{\otimes} g := a(\pi,Y,Z) \circ (\beta_Y \otimes \id_Z) \circ a(Y,\pi,Z)^{-1} \circ f \otimes g$.
\item If $f$ is odd and $g$ is even, let $f \widehat{\otimes} g := a(\pi,Y,Z) \circ (f \otimes g)$.
\item If $f$ and $g$ are both odd, let $f \widehat{\otimes} g$ be induced by the composition

\begin{equation*}
W \otimes X \xrightarrow{f \otimes g} \pi \otimes Y \otimes \pi \otimes Z \xrightarrow{\id_\pi \otimes \beta_Y \otimes \id_Z} \pi \otimes \pi \otimes Y \otimes Z \xrightarrow{-\xi \otimes \id_{Y \otimes Z}} Y \otimes Z
\end{equation*}
where we have suppressed associativity and unit isomorphisms for brevity.  

\end{itemize}
\end{definition}
One can check on a case by case basis that the composition defined in $\widehat{\uA}$ is associative.  The most interesting case is when $f \in \Hom_{\widehat{\uA}}(W,X)_1$, $g \in \Hom_{\widehat{\uA}}(X,Y)_1$ and $h \in \Hom_{\widehat{\uA}}(Y,Z)_1$
are homogeneous odd morphisms.  In this case, it suffices to show that the following diagram
\begin{center}
\includegraphics{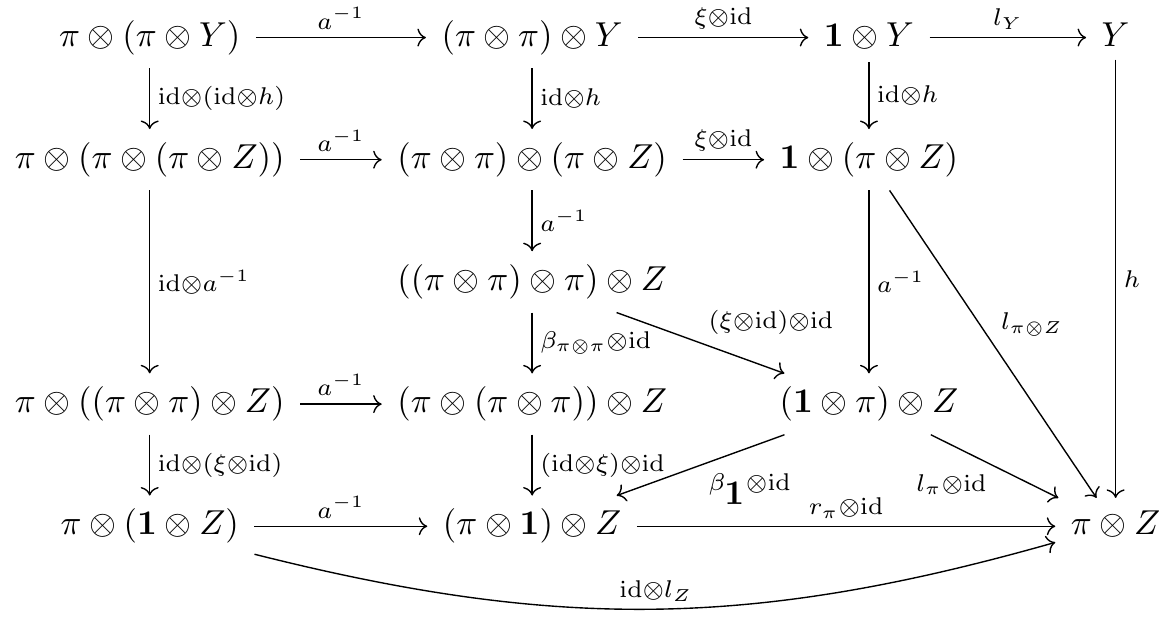}
\end{center}
is commutative.  To see this, observe that $\beta_{\pi \otimes \pi} = a(\pi,\pi,\pi)$ by Equation \eqref{eq:fcovergamma}, and so the 5-sided cell
commutes by the pentagon axiom.  In addition, we have used that $r_\pi \circ \beta_\unit = l_\pi$ for commutativity of the bottom right triangle.
All other cells commute by naturality or the triangle axiom.

It is a similar exercise to check that the tensor product defined in $\widehat{\uA}$ satisfies the super interchange law.  As before, the most interesting
case is when $f \in \Hom_{\widehat{\uA}}(W,Y)_1$, $g \in \Hom_{\widehat{\uA}}(X,Z)_1$, $h \in \Hom_{\widehat{\uA}}(A,W)_1$, and $k \in \Hom_{\widehat{\uA}}(B,X)_1$
are odd homogeneous morphisms, in which case we must show that $(f \widehat{\otimes} g) \widehat{\circ} (h \widehat{\otimes} k) = - (f \widehat{\circ} h) \widehat{\otimes} (g \widehat{\circ} k)$.  This reduces to showing that the following diagram

\begin{center}
\includegraphics{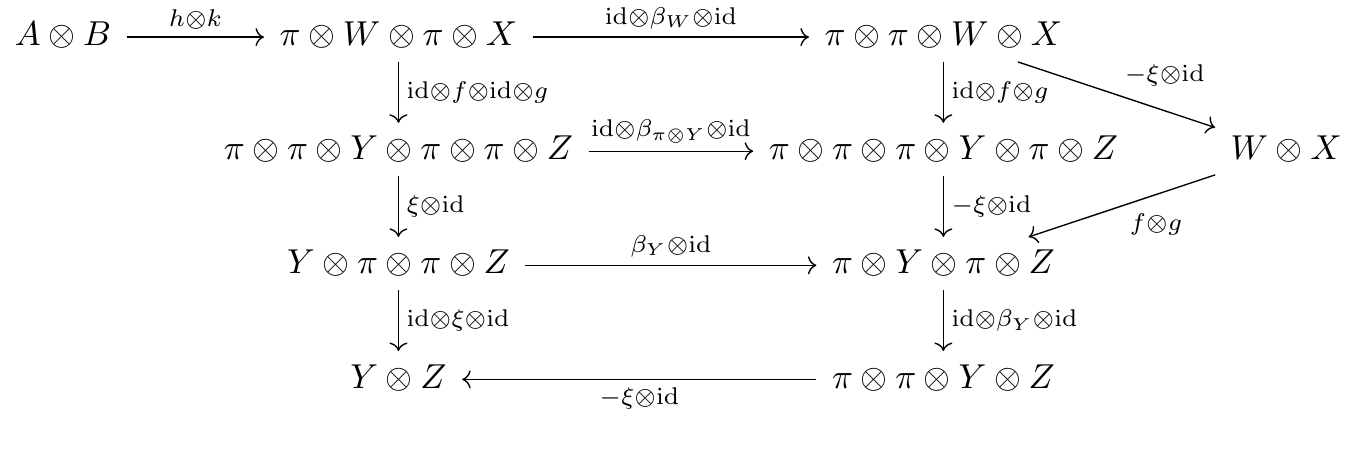}
\end{center}
is supercommutative. The top cell commutes by naturality of $\beta$, while the right cell commutes
by naturality of $-\xi$.  Since $(\pi,\beta)$ is in the centre $\mathcal{Z}(\uA)$, the diagram

\begin{center}
\includegraphics{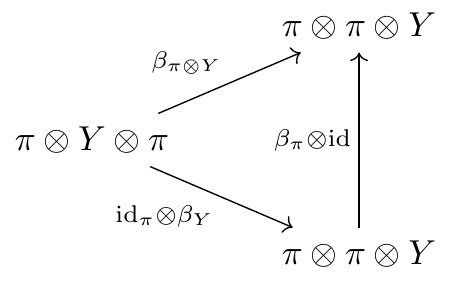}
\end{center}
is commutative.  Recalling that $\beta_\pi = -\id_{\pi \otimes \pi}$, we get that the middle cell commutes
by naturality of $\beta$.  The bottom cell supercommutes by comparison with Equation \eqref{eq:fcoverxi}.

\begin{lemma}[{\cite[Lemma 5.4]{BE:2016}}]
Let $(\uL,\pi,\zeta)$ be a $\Pi$-complete superfusion category, and define
\begin{align*}
G : \whuL &\to \uL \\
X &\mapsto X\\
f \in \Hom_{\whuL}(X,Y) &\mapsto \begin{cases}f \in \Hom_{\uL}(X,Y)_0 & \text{if } f \text{ even}\\
l_Y \circ (\zeta \otimes \id_Y) \circ f \in \Hom_{\uL}(X,Y)_1 & \text{if } f \text{ odd}
\end{cases}
\end{align*}
where $f$ is a homogeneous morphism.  Then $G$ is an isomorphism of superfusion categories.
\label{lem:assocup}
\end{lemma}

\begin{proof}
Observe that $G$ is a bijection on objects and morphisms, so it remains to show that $G$
respects composition and the tensor product.  Let $f \in \Hom_{\whuL}(X,Y)$ and $g \in \Hom_{\whuL}(Y,Z)$ be homogeneous morphisms in $\uL$.  We only consider the most interesting case when $f$ and $g$ are both odd.  In this case, $G(g) \circ G(f)$ is given by the composition

\[
X \xrightarrow{f} \pi \otimes Y \xrightarrow{\zeta \otimes \id_Y} \unit \otimes Y \xrightarrow{l_Y} Y \xrightarrow{g} \pi \otimes Z \xrightarrow{\zeta \otimes \id_Z} \unit \otimes Z \xrightarrow{l_Z} Z
\]
while $G(g \widehat{\circ} f)$ is given by the composition
\[
X \xrightarrow{f} \pi \otimes Y \xrightarrow{\id_\pi \otimes g} \pi \otimes (\pi \otimes Z) \xrightarrow{a(\pi,\pi,Z)^{-1}} (\pi \otimes \pi) \otimes Z \xrightarrow{\xi \otimes \id_Z} \unit \otimes \id_Z \xrightarrow{l_Z} Z
\]
and we must show these are equal.  Indeed, $g \circ l_Y \circ (\zeta \otimes \id_Y) = -l_{\pi \otimes Z} \circ (\zeta \otimes \id_{\pi \otimes Z}) \circ (\id_\pi \otimes g)$ by supernaturality, and so it remains to show that the following diagram

\begin{center}
\includegraphics{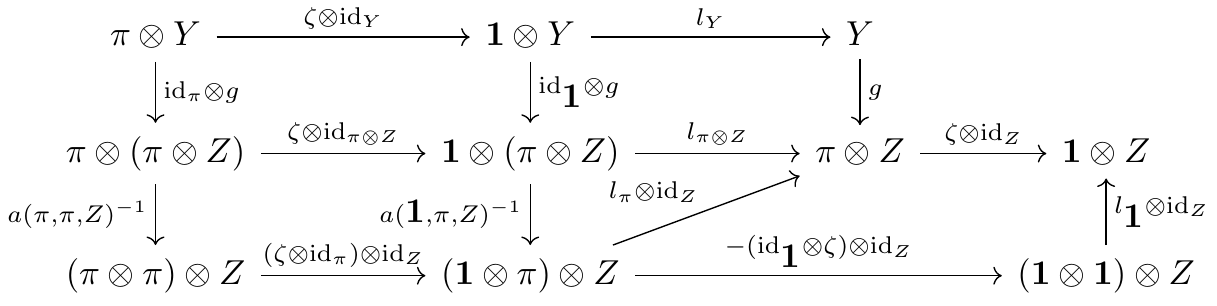}
\end{center}
is commutative, where we have used that $\xi = - l_\unit \circ (\id_\unit \otimes \zeta) \circ (\zeta \otimes \id_\pi)$.  The top left
cell supercommutes by supernaturality of $\zeta$, the top right cell commutes by naturality of $l$, the bottom left cell commutes by naturality of $a$, the bottom triangle commutes by the triangle axiom, and the bottom right cell supercommutes by naturality of $\zeta$.
Thus $G(g \widehat{\circ} f) = G(g) \circ G(f)$.

Similarly, one can check that if $f \in \Hom_{\whuL}(W,Y)$ and $g \in \Hom_{\whuL}(X,Z)$ then $G(f \widehat{\otimes} g) = G(f) \otimes G(g)$.  As before, we only consider the case where $f$ and $g$ are homogeneous odd morphisms.  In this case, we have
\begin{align*}
G(f) \otimes G(g) &= \left(l_Y \circ (\zeta \otimes \id_Y) \circ f\right) \otimes \left(l_Z \circ (\zeta \otimes \id_Z) \circ g\right)\\
									&= - (l_Y \otimes l_Z) \circ ((\zeta \otimes \id_Y) \otimes (\zeta \otimes \id_Z)) \circ (f \otimes g)
\end{align*}
by the super interchange law.  By comparing this with the definition of $G(f \widehat{\otimes} g)$, it suffices to show that the following diagram

\begin{center}
\includegraphics{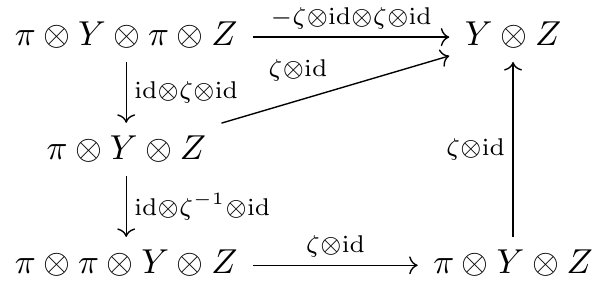}
\end{center}
is commutative, where we have again used that $-\xi = l_\unit \circ (\id_\unit \otimes \zeta) \circ (\zeta \otimes \id_\pi)$.
The top triangle supercommutes by the super interchange law, while the bottom cell supercommutes by supernaturality of $\zeta$, and so $G(f \widehat{\otimes} g) = G(f) \otimes G(g)$. 
\end{proof}

Thus every $\Pi$-complete superfusion category can be obtained from a fusion category over $\usvec$
by the above construction.  In fact, it is shown in \cite{BE:2016} that the functor $G$
 described above forms part of an equivalence between the category of fusion categories over $\usvec$ and the category
of $\Pi$-complete superfusion categories.

\section{The underlying fusion category}

\begin{definition}
Let $\uC$ be a superfusion category, and let $\euCpi$ be the underlying fusion category of the superadditive envelope of $\uC$ (see Definition \ref{def:pienvelope} and Definition \ref{def:saenvelope}).
We call $\euCpi$ the \emph{underlying fusion category} of $\uC$.
\end{definition}

In this section, we give an explicit formula for the 6$j$-symbols of $\euCpi$ in terms of the fermionic 6$j$-symbols of $\uC$.   Recall that for $X, Y \in \uC$ and $\underline{a}, \underline{b} \in \Z/2\Z$, we have
\[
\Hom_{\euCpi}(X^{\underline{a}},Y^{\underline{b}}) = \Hom_{\uC}(X,Y)_{\underline{a}+\underline{b}}
\]
If $f : X \to Y$ is a homogeneous morphism in $\uC$ and $\underline{a} + \underline{b} = |f|$,
then we denote by $f_{\underline{a}}^{\underline{b}}$ the corresponding morphism $X^{\underline{a}} \to Y^{\underline{b}}$ in $\euCpi$.
The tensor product of objects and morphisms in $\euCpi$ is defined by
\begin{align*}
X^{\underline{a}} \otimes Y^{\underline{b}} &:= (X \otimes Y)^{\underline{a}+\underline{b}}\\
f_{\underline{a}}^{\underline{b}} \otimes g_{\underline{c}}^{\underline{d}} &:= (-1)^{\underline{d}|f|} (f \otimes g)_{\underline{a}+\underline{c}}^{\underline{b}+\underline{d}}
\end{align*}

From Lemma \ref{lem:fcatusvec} we get that $(\euCpi, \ounit, \beta, \xi)$ is a fusion category over $\usvec$, where
\begin{equation*}
\beta_{\xa} = (-1)^{\underline{a}} \cdot (l_X^{-1} \circ r_X)_{\underline{a}+1}^{\underline{a}+1} : \xa \otimes \ounit \xrightarrow{\sim}  \ounit \otimes \xa,\ \xa \in \euCpi\end{equation*}

and
\begin{equation*}
\xi = (l_1)_{\underline{0}}^{\underline{0}} : \ounit \otimes \ounit \xrightarrow{\sim} \eunit
\end{equation*}
Let $X_i$, $i \in I$ be a set of representatives of isomorphism classes\footnote{We say that two objects in a superfusion category lie in the same isomorphism class if there is a (not necesssarily even) isomorphism between them.}
of simple objects in a superfusion
category $\uC$.  Define
\begin{equation}
J = \{ (i, \underline{a}) \in I \times \Z/2\Z \text{ such that } \underline{a} = 0 \text{ if $X_i$ is Majorana}\}
\label{eq:jlabel}
\end{equation}
We denote the element $(i,\underline{a}) \in J$ by $i^{\underline{a}}$.  The isomorphism
classes of simple objects in $\euCpi$ are labelled by $J$.  Indeed, suppose $X_i$ is Bosonic,
then we have a pair of non-isomorphic simple objects $X_i^{\underline{0}}$ and $X_i^{\underline{1}}$
in $\euCpi$ corresponding to the labels $i^{\underline{0}}$ and $i^{\underline{1}}$ respectively.  If $X_i$
is Majorana, then $X_i^{\underline{0}}$ and $X_i^{\underline{1}}$ are isomorphic in $\euCpi$, so
we choose $X_i^{\underline{0}}$ as our representative simple object, and label it by $i^{\underline{0}}$.

\begin{remark}
If $\uC$ is a Bosonic superfusion category, then the underlying fusion category $\euCpi$
has twice as many simple objects (up to isomorphism) as $\uC$, labelled by $J = I \times \Z/2\Z$.
\end{remark}

\begin{example} \label{ex:pted_2}
Let $\uC$ be a Bosonic pointed superfusion category, as in Example \ref{ex:pted_1}.  
The underlying fusion category $\euCpi$ is pointed, so let $G_\omega$ denote the (finite) group
of isomorphism classes of simple objects in $\euCpi$.  As a set, we have $G_\omega = \Z/2\Z \times G$, though we would like
to describe the group structure on $G_\omega$.  The isomorphisms $\e(g,h) : X_g \otimes X_h \iso X_{gh}$ in $\uC$ induce isomorphisms in $\euCpi$
\begin{equation*}
\e(g^{\underline{a}},h^{\underline{b}}) = \left(\e(g,h)\right)_{\underline{a}+\underline{b}}^{\underline{a}+\underline{b}+\omega(g,h)} : X_g^{\underline{a}} \otimes X_h^{\underline{b}} \iso X_{gh}^{\underline{a}+\underline{b}+\omega(g,h)}
\end{equation*}
for all $g^{\underline{a}}, h^{\underline{b}} \in G_\omega$, and so the group structure on $G_\omega$ is given by
\begin{equation*}
g^{\underline{a}} \cdot h^{\underline{b}} := (gh)^{\underline{a}+\underline{b}+\omega(g,h)}
\label{eq:pted_gpcentralextension}
\end{equation*}
and so $G_{\omega}$ is the central extension of $G$ by $\Z/2\Z$ determined by the 2-cocycle $\omega$.
\end{example}

\subsection{6j symbols in \texorpdfstring{$\euCpi$}{the underlying fusion category}}
Let $\uC$ be a superfusion category, and let $J$ label the simple objects in $\euCpi$, as described in \eqref{eq:jlabel}.  
Let $i^{\underline{a}}, j^{\underline{b}}, m^{\underline{c}} \in J$, and suppose that $(i,j,m,\alpha)$ is an admissable quadruple in $\uC$. 
If $\underline{c} = \underline{a} + \underline{b} + s^{ij}_m(\alpha)$ then $\e^{ij}_m : X_i \otimes X_j \to X_m$
induces a morphism
\[
X_i^{\underline{a}} \otimes X_j^{\underline{b}} \to X_m^{\underline{c}}
\]
in $\euCpi$, in which case $(i^{\underline{a}},j^{\underline{b}},m^{\underline{c}},\alpha)$ is an admissable quadruple in $\euCpi$. This implies that every admissable quadruple in $\euCpi$ 
can be written unambiguously in the form $$(i^{\underline{a}},j^{\underline{b}},m,\alpha)$$ where $i^{\underline{a}}, j^{\underline{b}},m^{\underline{a}+\underline{b}+s^{ij}_m(\alpha)} \in J$
and $(i,j,m,\alpha)$ is an admissable quadruple in $\uC$.  In the same way, every admissable decuple in $\euCpi$ can be written unambiguously as
\begin{equation*}
(i^{\underline{a}},j^{\underline{b}}, m, k^{\underline{c}},n,t,\alpha,\beta,\eta,\varphi)
\label{eq:eucpidecuple}
\end{equation*}
where $i^{\underline{a}}$, $j^{\underline{b}}$, $m^{\underline{a}+\underline{b}+s^{ij}_m(\alpha)}$, 
$t^{\underline{b}+\underline{c}+s^{jk}_t(\eta)}$, $n^{\underline{a}+\underline{b}+\underline{c} + s^{ij}_m(\alpha) + s^{mk}_n(\beta)} \in J$,
and $(i,j,m,k,n,t,\alpha,\beta,\eta,\varphi)$ is a parity admissable decuple in $\uC$. 

\begin{definition} \label{thm:6jrelation} Let $\uC$ be a superfusion category, and $\euCpi$ its underlying fusion category.  
If $(i^{\underline{a}},j^{\underline{b}}, m, k^{\underline{c}},n,t,\alpha,\beta,\eta,\varphi)$ is an admissable
decuple in $\euCpi$, let
\begin{equation*}
F^{i^{\underline{a}}j^{\underline{b}}m,\alpha\beta}_{k^{\underline{c}}nt,\eta\varphi} := (-1)^{\underline{c} s^{ij}_m(\alpha)} \widetilde{F}^{ijm,\alpha\beta}_{knt,\eta\varphi}.
\end{equation*}
If $(i^{\underline{a}},j^{\underline{b}}, m, k^{\underline{c}},n,t,\alpha,\beta,\eta,\varphi)$ is not admissable,
then let $F^{i^{\underline{a}}j^{\underline{b}}m,\alpha\beta}_{k^{\underline{c}}nt,\eta\varphi} = 0$.
\end{definition}

We claim that the symbols defined above are in fact the 6$j$-symbols of $\euCpi$.  Indeed,
they satisfy the following version of the pentagon equation.

\begin{theorem}[Pentagon equation]\label{thm:peq}
Let $\uC$ be a superfusion category with simple objects indexed by a set $I$, and $\euCpi$
the underlying fusion category.  Then
\begin{multline}
\sum_{t \in I} \sum_{\eta = 1}^{N_t^{jk}} \sum_{\varphi = 1}^{N_n^{it}} \sum_{\kappa=1}^{N_s^{tl}}  F^{i^{\underline{a}}j^{\underline{b}}m,\alpha\beta}_{k^{\underline{c}}nt,\eta\varphi} F^{i^{\underline{a}}t^{\underline{b}+\underline{c}+s^{jk}_t(\eta)}n,\varphi\chi}_{l^{\underline{d}}ps,\kappa\gamma} F^{j^{\underline{b}} k^{\underline{c}} t, \eta \kappa}_{l^{\underline{d}}sq,\delta\phi} \\= \sum_{\epsilon=1}^{N_p^{mq}} F^{m^{\underline{a}+\underline{b}+s^{ij}_m(\alpha)} k^{\underline{c}} n, \beta \chi}_{l^{\underline{d}} p q, \delta \epsilon} F^{i^{\underline{a}}j^{\underline{b}}m,\alpha \epsilon}_{q^{\underline{c}+\underline{d}+s^{kl}_q(\delta)}ps,\delta \gamma}
\label{eq:peftilde}
\end{multline}
for all $i,j,k,l,m,n,t,p,q,s \in I$, $\underline{a}, \underline{b}, \underline{c} \in \Z/2\Z$, and $\alpha, \beta, \eta, \chi, \delta, \phi \in \Z_{\geq0}$.
\end{theorem}

\begin{proof}
By combining our definition \eqref{thm:6jrelation} with the super pentagon equation \eqref{def:6j_spentagon}, we have the equality
\begin{multline*}
\sum_{t \in I} \sum_{\eta=1}^{N_t^{jk}} \sum_{\varphi=1}^{N_n^{it}} \sum_{\kappa=1}^{N_s^{tl}} (-1)^{\underline{c}s^{ij}_m(\alpha)+\underline{d}s^{it}_n(\varphi)+\underline{d} s^{jk}_t(\eta)} F^{ijm,\alpha\beta}_{knt,\eta\varphi} F^{itn,\varphi\chi}_{lps,\kappa\gamma} F^{jkt,\eta\kappa}_{lsq,\delta\phi}\\
= (-1)^{s^{ij}_m(\alpha)s^{kl}_q(\delta)} \sum_{\epsilon=1}^{N^{mq}_p} (-1)^{\underline{d}s^{mk}_n(\beta)+(\underline{c}+\underline{d}+s^{kl}_q(\delta))s^{ij}_m(\alpha)} F^{mkn,\beta\chi}_{lpq,\delta\epsilon} F^{ijm,\alpha\epsilon}_{qps,\delta\gamma}
\end{multline*}
and thus it suffices to show that
\begin{equation*}
\underline{c}s^{ij}_m(\alpha) + \underline{d}s^{it}_n(\varphi) + \underline{d} s^{jk}_t(\eta) = s^{ij}_m(\alpha)s^{kl}_q(\delta) + \underline{d} s^{mk}_n(\beta) + (\underline{c}+\underline{d}+s^{kl}_q(\delta))s^{ij}_m(\alpha)
\end{equation*}
for all admissable decuples $(i^{\underline{a}},j^{\underline{b}},m,k^{\underline{c}},n,t,\alpha,\beta,\eta,\varphi)$ in $\euCpi$.  This immediately
reduces to showing that
\begin{equation*}
\underline{d} s^{it}_n(\varphi) + \underline{d} s^{jk}_t(\eta) = \underline{d} s^{mk}_n(\beta) + \underline{d} s^{ij}_m(\alpha)
\end{equation*}
which holds by the parity compatibility condition \eqref{eq:defparityadmissable}.
\end{proof}

\begin{remark}
Our definition  of the 6$j$-symbols in $\euCpi$ can be recovered directly from
the construction of $\euCpi$, in which case Theorem \ref{thm:peq} can be viewed as a corollary of the pentagon axiom in $\euCpi$.
Indeed, for each admissable quadruple $(i^{\underline{a}},j^{\underline{b}},m,\alpha)$ in $\euCpi$, let
\begin{equation}
\e^{i^{\underline{a}}j^{\underline{b}}}_m(\alpha) := \left(\e^{ij}_m(\alpha)\right)_{\underline{a}+\underline{b}}^{\underline{a}+\underline{b}+s^{ij}_m(\alpha)} : X_i^{\underline{a}} \otimes X_j^{\underline{b}} \to X_m^{\underline{a}+\underline{b}+s^{ij}_m(\alpha)}
\end{equation}
For ease of notation, set $\underline{d} = \underline{a}+\underline{b}+s^{ij}_m(\alpha)$ and $\underline{e} = \underline{a}+\underline{b}+ \underline{c} s^{ij}_m(\alpha) + s^{mk}_n(\beta)$,
then \eqref{eq:6j_fusioncat} is given by
\begin{equation}
(X_i^{\underline{a}} \otimes X_j^{\underline{b}}) \otimes X_k^{\underline{c}} \xrightarrow{\e^{i^{\underline{a}}j^{\underline{b}}}_m(\alpha) \otimes \id_{X_k^{\underline{c}}}} X_m^{\underline{d}} \otimes X_k^{\underline{c}} \xrightarrow{\e^{m^{\underline{d}}k^{\underline{c}}}_n(\beta)} X_n^{\underline{e}}
\label{eq:uelhs}
\end{equation}
where we have
\begin{equation}
\e^{i^{\underline{a}}j^{\underline{b}}}_m(\alpha) \otimes \id_{X_k^{\underline{c}}} = (-1)^{\underline{c} s^{ij}_m(\alpha)} \left(\e^{ij}_m(\alpha) \otimes \id_{X_k}\right)_{\underline{a}+\underline{b}+\underline{c}}^{\underline{c}+\underline{d}}
\end{equation}
by definition of the tensor product on $\euCpi$.  Next, fix an admissable quadruple $(j^{\underline{b}},k^{\underline{c}}, t, \eta)$.
The composition \eqref{eq:6j_fusioncatrhs} is given by
\begin{equation}
(X_i^{\underline{a}} \otimes X_j^{\underline{b}}) \otimes X_k^{\underline{c}} \xrightarrow{a(X_i^{\underline{a}},X_j^{\underline{b}},X_k^{\underline{c}})} X_i^{\underline{a}} \otimes (X_j^{\underline{b}} \otimes X_k^{\underline{c}}) \xrightarrow{ \id_{X_i^{\underline{a}}} \otimes \e^{j^{\underline{b}}k^{\underline{c}}}_{t}(\eta)} X_i^{\underline{a}} \otimes X_t^{\underline{f}} \xrightarrow{\e^{i^{\underline{a}}t^{\underline{f}}}_n(\varphi)} X_n^{\underline{e}}
\label{eq:uerhs}
\end{equation}
where $\underline{f} = \underline{b}+\underline{c}+s^{jk}_t(\eta)$.  We compute
\begin{equation}
\id_{X_i^{\underline{a}}} \otimes \e^{j^{\underline{b}}k^{\underline{c}}}_{t}(\eta) = \left( \id_{X_i} \otimes \e^{jk}_t(\eta) \right)_{\underline{a}+\underline{b}+\underline{c}}^{\underline{a}+\underline{f}}
\end{equation}
and so the compositions \eqref{eq:uelhs} and \eqref{eq:uerhs} in $\euCpi$ are induced by the corresponding compositions \eqref{eq:6j_fusioncat} and \eqref{eq:6j_fusioncatrhs} in $\uC$ up to a factor of $(-1)^{\underline{c}s^{ij}_m(\alpha)}$, as expected.
\end{remark}

\begin{example}
Let $\uC$ be a Bosonic pointed superfusion category, as in Examples \ref{ex:pted_1} and \ref{ex:pted_2}.  
For all $g^{\underline{a}}, h^{\underline{b}}, k^{\underline{c}} \in G_\omega$ we can unambiguously write $F(g^{\underline{a}},h^{\underline{b}},k^{\underline{c}}) \in \C^\times$ for the corresponding 6$j$-symbol in $\euCpi$.
With this notation Theorem \ref{thm:6jrelation} implies 
\begin{equation*}
F(g^{\underline{a}},h^{\underline{b}},k^{\underline{c}}) = (-1)^{\underline{c} \omega(g,h)} \widetilde{F}(g,h,k) 
\end{equation*}
for all $g^{\underline{a}}, h^{\underline{b}}, k^{\underline{c}} \in G_\omega$.  The pentagon equation \eqref{eq:peftilde} implies
that $F$ is a 3-cocycle on $G_\omega$ with values in $\C^\times$.
\end{example}

Viewing $G$ as the subset of $G_\omega$ consisting of elements of the form $g^{\underline{0}}$,
we have the following corollary.
\begin{corollary}
Let $\widetilde{F} : G^3 \to \C^\times$ be a 3-supercocycle on $G$ with 2-cocycle $\omega$.
Then there exists a 3-cocycle $F : G_\omega^3 \to \C^\times$ on $G_\omega$ such that
\begin{equation*}
F|_{G^3} = \widetilde{F}
\end{equation*}
\label{cor:super3cocycle}
\end{corollary}

In other words, every 3-supercocycle on $G$ arises as the restriction of a 3-cocycle on
a central extension of $G$ by $\Z/2\Z$.

\section{Applications}
In this section, we describe some applications of the theory of fusion categories
to that of superfusion categories.  In particular, we define the $\pi$-Grothendieck ring of a superfusion category,
and prove a version of Ocneanu rigidity for superfusion categories.

\subsection{Superforms}

Let $\uD$ be a $\Pi$-complete superfusion category.  A \emph{superform} of $\uD$ is a superfusion category
$\uC$ such that $\uC \simeq \uD$ are equivalent (but not necessarily superequivalent) 
superfusion categories.  Our goal is to prove the following.

\begin{proposition}
A $\Pi$-complete superfusion category $\uD$ has only finitely many superforms, up to superequivalence of superfusion categories.
\label{prop:finsuperforms}
\end{proposition}

To show this, the following notion will be useful.

\begin{definition}
Let $\uC$ and $\uD$ be superfusion categories, and $F : \uC \to \uD$ a tensor superfunctor.
Its \emph{even essential image} $F(\uC)$ is the full subcategory of $\uD$ consisting of objects
evenly isomorphic to $F(X)$ for some $X \in \uC$.
\end{definition}

Recall that a tensor superfunctor $F : \uC \to \uD$ is a superfunctor such
that $F(\unit_{\uC})$ is evenly isomorphic to $\unit_{\uD}$, together with an even
natural isomorphism $c_{X,Y} : F(X) \otimes F(Y) \iso F(X \otimes Y)$ satisfying the usual diagram (see \cite[\S 2.4]{ENO:2015}).  Observe that $F(\uC)$ is a full tensor subcategory of $\uD$.  Indeed given $Y, Y' \in F(\uC)$,
there exists $X, X' \in \uC$ such that $F(X) \iso Y$ and $F(X') \iso Y'$ are evenly isomorphic.
Then $F(X \otimes X') \iso F(X) \otimes F(X') \iso Y \otimes Y'$ is an even isomorphism, whence
$Y \otimes Y' \in F(\uC)$.  

\begin{lemma}
If $F : \uC \to \uD$ is an equivalence of superfusion categories, then $\uC$ is determined (up to superequivalence)
by $F(\uC)$.  More precisely, if $G : \uA \to \uD$ is an equivalence of superfusion categories with $G(\uA) = F(\uC)$,
then $\uA$ and $\uC$ are superequivalent superfusion categories.
\label{lem:superformequiv}
\end{lemma}

\begin{proof}
If $X \in \uA$, then $G(X) \in G(\uA) = F(\uC)$, so there exists
$X_{\uC} \in \uC$ such that $F(X_{\uC}) \iso G(X)$ are evenly isomorphic.  For each $X \in \uA$,
we pick such a $X_{\uC} \in \uC$ together with an even isomorphism $q_X : F(X_{\uC}) \iso G(X)$.  We define a superfunctor $K : \uA \to \uC$ as follows.  On objects, let $K(X) = X_{\uC}$.
On morphisms, if $f \in \Hom_{\uA}(X,Y)$ then let $K(f) = F^{-1}(q_Y^{-1} \circ G(f) \circ q_X)$, i.e.
$K(f)$ is the image of $f$ under the even isomorphism
\[
\Hom_{\uA}(X,Y) \xrightarrow{G} \Hom_{\uD}(G(X), G(Y)) \xrightarrow{(q_Y^{-1})_* \circ (q_X)^*} \Hom_{\uC}(F(X_\uC),F(Y_{\uC})) \xrightarrow{F^{-1}} \Hom_{\uC}(X_{\uC},Y_{\uC})
\]
It is immediate from $K$ is fully faithful, and functorality of $F$ and $G$ implies $K$ is a superfunctor.  
To show that $K$ is a superequivalence, we must show that $K(\uA) = \uC$.  Let $Y \in \uC$, then $F(Y) \in F(\uC) = G(\uA)$ so there exists
$X \in \uA$ together with an even isomorphism $G(X) \iso F(Y)$, so $F(X_c) \iso F(Y)$ are evenly isomorphic.
This implies that $K(X) = X_c \iso Y$ are evenly isomorphic, i.e. $Y \in K(\uA)$. Thus $K$ is a superequivalence.

It remains to endow $K$ with the structure of a monoidal superfunctor.  To do this, we must define
even coherence maps $J_{X,Y} : K(X) \otimes K(Y) \to K(X \otimes Y)$ satisfying the usual axioms.  Let $c$ and $d$ denote the coherence maps for $F$ and $G$ respectively.  Let ${\varphi_{X,Y} : F(X_\uC \otimes Y_\uC) \iso F((X \otimes Y)_\uC)}$
be the composition
\[
F(X_\uC \otimes Y_\uC) \xrightarrow{c_{X_\uC,Y_\uC}^{-1}} F(X_\uC) \otimes F(Y_\uC) \xrightarrow{q_x \otimes q_y} G(X) \otimes G(Y) \xrightarrow{d_{X,Y}} G(X \otimes Y) \xrightarrow{q_{X \otimes Y}^{-1}} F((X \otimes Y)_\uC)
\]
With this notation, let $J_{X,Y} := F^{-1}(\varphi_{X,Y})$.  It is straightforward to check that $(K,J)$ satisfies
the axioms for a monoidal superfunctor, and so $K$ is a superequivalence of superfusion categories.
\end{proof}

We are now ready to prove the above proposition.

\begin{proof}[Proof of Proposition \ref{prop:finsuperforms}]
Let $F : \uC \to \uD$ be an equivalence of superfusion categories, where $\uD$ is $\Pi$-complete.  Let $Y_i$, $i \in I$ be a set of representatives of simple objects of $\uD$.  Since $F$ is an equivalence,
there exists $X_i$, $i \in I$ such that $F(X_i) \iso Y_i$.  Since $\uD$ is $\Pi$-complete, for each
$i \in I$ there exists $Y'_i \in \uD$ such that $Y_i \iso Y'_i$ are oddly isomorphic.   Fix $i \in I$.  If $Y_i$ is Majorana, then $\Hom_{\uD}(F(X_i), Y_i) \simeq \C^{1|1}$, so $Y_i \in F(\uC)$
and $Y'_i \in F(\uC)$.  If $Y_i$ is Bosonic, then the space $\Hom_{\uD}(F(X_i),Y_i)$ is one-dimensional,
either even or odd.  So $Y_i \in F(\uC)$ or $Y'_i \in F(\uC)$ (or possibly both).  Since the subcategory $F(\uD)$
is determined by the choice of $Y_i$ or $Y_i'$ (or both) for all $i \in I$ such that $Y_i$ is Bosonic,
and there are finitely many such choices, there are finitely many possibilities for $F(\uC)$.  By Lemma \ref{lem:superformequiv}, we are done.
\end{proof}

\begin{corollary}
The number of superfusion categories (up to superequivalence) is countable.
\end{corollary}

\begin{proof}
Ocneanu rigidity \cite[Theorem 2.28, Theorem 2.31]{ENO:OFC} implies there are countably many fusion categories
over $\usvec$.  Since every $\Pi$-complete superfusion category
is the associated superfusion category of a fusion category over $\usvec$, there are countably
many $\Pi$-complete superfusion categories.  Every superfusion category is equivalent to a $\Pi$-complete
superfusion category, so Proposition \ref{prop:finsuperforms} implies the result.
\end{proof}

\subsection[Pi-Grothendieck ring]{$\pi$-Grothendieck ring}
Let $\Z^\pi = \Z[\pi]/(\pi^2 - 1)$ and $\Z^\pi_+ = \{a + b \pi\ :\ a, b \in \Z_{\geq 0}\} \subset \Z^\pi$.

\begin{definition}
The \emph{$\pi$-Grothendieck group} of a supercategory $\uC$ is the $\Z^\pi$-module $\sgr$ generated by isomorphism classes of objects $[X]$ in $\uC$ subject to the relation that if $0 \to X \xrightarrow{f} Y \xrightarrow{g} Z \to 0$ is a short exact
sequence with $f$ and $g$ homogeneous morphisms, then $[Y] = [X] \pi^{|f|} + [Z] \pi^{|g|}$.  
\end{definition}

If $\uC$ is a rigid monoidal supercategory, then the tensor product on $\uC$ induces an associative multiplication
on $\sgr$, given by $[X] \cdot [Y] := [X \otimes Y]$, making $\sgr$ into a $\Z^\pi$-algebra.

\begin{definition}
We call $\sgr$ the \emph{$\pi$-Grothendieck ring} of $\uC$.
\end{definition}

\begin{example} \label{ex:svecfinsgrr}
Let $\svec_{\mathrm{fin}}$ denote the monoidal supercategory of finite dimensional superspaces together
with all linear maps between them, and let $\C^{p|q} = \C^p \oplus \C^q$ denote the superspace with $(\C^{p|q})_0 = \C^p$ and $(\C^{p|q})_1 = \C^q$,
then
\[
[\C^{p|q}] = p [\C^{1|0}] + q[\C^{0|1}] = (p + q\pi) [\C^{1|0}]
\]
in $\sgrr(\svec_{\mathrm{fin}})$, where we used that $[\C^{0|1}] = \pi [\C^{0|1}]$.  Since every object in $\svec_{\mathrm{fin}}$
is evenly isomorphic to $\C^{p|q}$ for some $p$ and $q$, this implies that $\sgrr(\svec_{\mathrm{fin}})$ is a free $\Z^\pi$-module,
generated by $[\C^{1|0}]$.  Moreover, the tensor product on $\svec_{\mathrm{fin}}$ gives
\[
[\C^{p|q}] [\C^{p'|q'}] = [\C^{pp'+qq'|pq'+qp'}]
\]
and so $\sgrr(\svec_{\mathrm{fin}})$ is free as a $\Z^\pi$-algebra.
\end{example}

Let $\uC$ be a superfusion category, and $X_i$, $i \in I$ representatives of the isomorphism classes
of simple objects in $\uC$.  To each $X$ in $\uC$ we can canonically associate the class $[X] \in \sgr$ given by the formula
\begin{equation}
[X] = \sum_i [X:X_i] [X_i]
\label{pgr:canonicalclass}
\end{equation}
where 
\begin{equation}
[X : X_i] = \dim \Hom_{\uC}(X_i,X)_{0} + \pi \dim \Hom_{\uC}(X_i,X)_1 \in \Z^\pi
\label{pgr:multiplicity}
\end{equation}
is the multiplicity of $X_i$ in $X$.  The multiplication on $\sgr$ is defined by
\[
[X_i] \cdot [X_j] = \sum_k [X_i \otimes X_j : X_k] [X_k]
\]

\begin{example}
Let $\uI$ be an \emph{Ising braided category}, i.e. a braided fusion category 
with $\mathrm{FPdim}(\uI) = 4$ that is not pointed, see \cite[Appendix B]{DGNO:BFC}.
Such a category contains precisely 3 isomorphism classes of simple objects: the unit object
$\unit$, an invertible object $\pi$ and a non-invertible object $X$ satisfying the fusion rules:
\[
\pi \otimes \pi \simeq \unit,\ \quad \pi \otimes X \simeq X \simeq X \otimes \pi,\ \quad X \otimes X \simeq \unit \oplus \pi
\]
The fusion subcategory $\uI_{ad} \subset \uI$ generated by $\unit$ and $\pi$ is braided equivalent to $\usvec$ \cite[Lemma B.11]{DGNO:BFC},
and thus $\uI$ is a fusion category over $\usvec$.  Let us consider the associated superfusion
category $\widehat{\uI}$.

The isomorphism $\pi \otimes \pi \simeq \unit$ in $\uI$ induces an odd isomorphism $\pi \iso \unit$ in $\widehat{\uI}$.
Similarly, the isomorphism $\pi \otimes X \simeq X$ in $\uI$ induces an odd isomorphism $X \iso X$ in $\widehat{\uI}$.
Thus $\widehat{\uI}$ has a Bosonic simple object $\unit \iso \pi$, and a Majorana simple object $X$.  From the fusion rules, we get the relations
\[
[X] = \pi[X],\quad \ [X]^2 = (1 + \pi)[1]
\]
in $\sgrr(\widehat{\uI})$.
\end{example}

\begin{example}[{see \cite[\S 8.18.2]{ENO:2015}}]
Generalising the previous example, take $k \equiv 2$ mod 4, and let $C_k(q)$ denote the braided fusion category
of integrable $\widehat{\mathfrak{sl}_2}$ modules at level $k$. This category has simple objects $V_i$, $i = 0,\ldots,k$
with unit object $V_0 = \unit$ and fusion rule given by the truncated Clebsch-Gordan rule:
\begin{equation}
V_i \otimes V_j \simeq \bigoplus_{l=\max(i+j-k,0)}^{\min(i,j)} V_{i+j-2l}
\label{eq:tcgr}
\end{equation}
The fusion subcategory $D_k(q) \subset C_k(q)$ generated by $\unit$ and $\pi := V_k$
is braided equivalent to $\usvec$, and so $C_k(q)$ is a fusion category over $\usvec$.
Let $\uC_k := \widehat{C_k(q)}$ denote the associated superfusion category.

Since $\pi \otimes V_i \simeq V_{k-i}$ in $C_k(q)$ for all $i = 0,\ldots,k$, we
have $V_i \iso V_{k-i}$ in $\uC_k$.  Thus $C_k(q)$ has $k/2$ Bosonic simple objects
$V_0, V_1,\ldots,V_{k/2-1}$, and a single Majorana simple object $V_{k/2}$.
\end{example}

Finally, we arrive at the following version of Ocneanu rigidity for superfusion categories.

\begin{corollary}
The number of superfusion categories (up to superequivalence) with a given $\pi$-Grothendieck ring
is finite.
\end{corollary}

\begin{proof}
Fix a superfusion category $\uC$, and suppose that $\uD$ is a superfusion category with $\sgrr(\uC) \simeq \sgrr(\uD)$.
We will show that there are finitely many possibilities for $\uD$, up to superequivalence.  Since
$\sgrr(\uC) \simeq \sgrr(\uD)$, the underlying fusion categories $\euCpi$ and $\underline{\uD_\pi^+}$ have
isomorphic Grothendieck rings.  By Ocneanu rigidity \cite[Theorem 2.28]{ENO:OFC}, there are finitely many fusion categories with a given Grothendieck ring,  and moreover each of these fusion categories $\uA$ admits only finitely many tensor functors
$\usvec \to \mathcal{Z}(\uA)$ \cite[Theorem 2.31]{ENO:OFC}, hence there are finitely many
fusion categories over $\usvec$ with Grothendieck ring isomorphic to $\mathrm{Gr}(\euCpi)$. Lemma \ref{lem:assocup}
then implies that there are finitely many possibilties for $\uD_\pi^+$ up to superequivalence,
so by Proposition \ref{prop:finsuperforms} there are finitely many possibilities for $\uD$ up to superequivalence.
\end{proof}

\bibliography{references}{}
\bibliographystyle{alpha}

\end{document}